\documentclass{amsart}

\usepackage[english]{babel}
\usepackage[latin1]{inputenc}
\usepackage[dvips,final]{graphics}
\usepackage{amsmath,amsfonts,amssymb,amsthm,amscd,array,stmaryrd,mathrsfs}
\usepackage{pstricks}
 \usepackage[all]{xy}
\usepackage{textcomp}
 \usepackage[final]{epsfig}
\theoremstyle{plain}
\newtheorem{thm}{Theorem}
\newtheorem{lem}{Lemma}[section]

\newtheorem{prop}[lem]{Proposition}

\theoremstyle{definition}
\newtheorem{defn}[lem]{Definition}
\newtheorem{rem}[lem]{Remark}
\newtheorem{ex}[lem]{Example}

\newcommand{\Z}{\mathbb{Z}}
\newcommand{\C}{\mathbb{C}}
\newcommand{\N}{\mathbb{N}}

\newcommand{\Pb}{\mathbb{P}}

\newcommand{\Ab}{\mathbf{A}}

\newcommand{\jtp}{\mathcal{J}_{\sigma}(\theta,p)}

\newcommand{\Sc}{\mathcal{S}}

\newcommand{\Kc}{\mathcal{K}}
\newcommand{\Ac}{\mathcal{A}}

\newcommand{\Jc}{\mathcal{J}}

\newcommand{\Ic}{\mathcal{I}}

\newcommand{\F}{\mathscr{F}}

\newcommand{\Vc}{\mathscr{V}}

\newcommand{\lf}{\mathfrak{v}}
\newcommand{\Lf}{\mathfrak{L}}

\newcommand{\Wf}{\mathfrak{W}}
\newcommand{\Kf}{\mathfrak{K}}
\newcommand{\Of}{\mathfrak{o}}
\newcommand{\Sf}{\mathfrak{s}}

\newcommand{\gf}{\mathfrak{g}}

\newcommand{\half}{\frac{1}{2}}

\newcommand{\quart}{\textstyle \frac{1}{4}}
\newcommand{\End}{\textup{\bf End}}
\newcommand{\Hom}{\textup{Hom}}

\newcommand{\e}{\varepsilon}

\newcommand{\asl}{\mathrm{asl}}

\newcommand{\Der}{\textup{\bf Der}}

\newcommand{\ddz}{\frac {d}{dz}} 

\def\a{\alpha}

\def\d{\delta}
\def\e{\varepsilon}

\def\g{\gamma}

\def\om{\omega}
\def\r{\rho}
\def\s{\sigma}
\def\t{\theta}
\def\vp{\varphi}

\def\l{\lambda}

\def\Si{\Sigma}

\begin{document}

\title[Superalgebras associated to Riemann surfaces]
{Superalgebras associated to Riemann surfaces:
Jordan algebras of Krichever-Novikov type
}

\author{S{\'e}verine Leidwanger
and  Sophie Morier-Genoud}

\address{
S{\'e}verine Leidwanger,
Institut Math\'e\-ma\-tiques de Jussieu,
Th\'eorie des groupes,
Universit\'e Denis Diderot Paris VII,
case 7012,
75205 Paris Cedex 13, France}

\email{leidwang@math.jussieu.fr}

\address{
Sophie Morier-Genoud,
Institut Math\'e\-ma\-tiques de Jussieu,
UMR 7586,
Universit\'e Pierre et Marie Curie Paris VI,
4 place Jussieu, case 247,
75252 Paris Cedex 05, France}

\email{sophiemg@math.jussieu.fr}

\date{}


\begin{abstract}
We construct two superalgebras associated to a punctured Riemann surface.
One of them is a Lie superalgebra of Krichever-Novikov type,
the other one is a Jordan superalgebra.
The constructed algebras are related in several ways (algebraic, geometric, representation theoretic).
In particular, the Lie superalgebra is the algebra of derivations of the
Jordan superalgebra.
\end{abstract}

\maketitle

\section{Introduction}

In 1987, Krichever and Novikov \cite{KN1}, \cite{KN2}, \cite{KN3} introduced a family of Lie algebras generalizing the Virasoro algebra. 
Given a Riemann surface of arbitrary genus, 
the Krichever-Novikov algebra is the algebra of meromorphic vector fields on the surface which are holomorphic outside 
two distinguish fixed points. This algebra admits non-trivial central extensions.
The case where the Riemann surface is the sphere corresponds exactly to the Virasoro algebra.
Later, this definition has been extended to the case of graded Riemann surfaces, \cite{BMRR}, \cite{Bry1}, \cite{Bry2}, and also to the case of Riemann surfaces punctured by more than two points, 
\cite{Dic}, \cite{Schli2}.

In the present paper, we study two natural superalgebras, $\Lf_{KN}$ and $\Jc_{KN}$,
coming from punctured Riemann surfaces.
One of them, $\Lf_{KN}$, has a structure of Lie superalgebra. 
It is constructed from the natural action of the algebra of meromorphic vector fields
on the space of half densities. 
The other one, $\Jc_{KN}$, is a commutative superalgebra, which enters the class of Jordan superalgebra.
It is constructed from the natural action of the algebra of meromorphic functions
on the space of half densities.

One of the main notion used in the paper is that of \textit{Lie antialgebras},
introduced in 2007 by V.Ovsienko  \cite{Ovs}.
This class of algebras is a subclass of Jordan superalgebras.
Ovsienko explained how one can associate a Lie superalgebra to a Lie antialgebra
(the process is different from the one of Koecher-Kantor-Tits) and how the representations of these algebras are related.

It turns out that the algebra $\Jc_{KN}$ that we introduce is a Lie antialgebra.
Our first goal is to understand the relation between the algebras $\Lf_{KN}$ and $\Jc_{KN}$.
Theorem~\ref{KNlink} establishes two different links between the two algebras:
the first link within the framework
of Lie anti\-algebras, and a second geometric link in terms of algebras of derivations.
The next main result of the paper, Theorem \ref{goodrep}, 
provides a classification of representations of
$\Jc_{KN}$ araising from tensor densities modules of $\Lf_{KN}$.

Our running example is the case of the Riemann surface of genus 0 with 3 punctures.
It turns out that the algebra $\Jc_{KN}$ that we obtain in this case is similar to the ones considered in \cite{Zhe1}, \cite{Zhe2} as a new type of infinite dimensional Jordan superalgebra. Section \ref{algeb} and Theorem \ref{ThmSubalg} give an algebraic construction of the algebra $\Jc_{KN}$
leading to a connection with the work of \cite{Zhe1}, \cite{Zhe2}.


\section{Preliminary}\label{prem}

In this section, we recall briefly the main notions related to Lie superalgebras and Jordan superalgebras. 
We refer to \cite{Martinez} (and the references therein) for more general theory of these structures.
We also recall the definitions concerning Lie antialgebras \cite{Ovs}.

The algebras are considered over the field of complex numbers $\C$ 
(although most of the notions make sense over any field of characteristic not 2).
For a homogeneous element $v$ in a $\Z_2$-graded vector space $V=V_0\oplus V_1$ 
we denote $\bar{v}$ the degree of $v$, i.e.   $\bar{v}=i$ for $v\in V_i$.
In $\End(V_0\oplus V_1)$, the even elements, resp. odd elements,  are those morphisms belonging to 
$\End(V_0)\oplus \End(V_1)$, resp. $\Hom(V_0,V_1)\oplus\Hom(V_1,V_0)$.

\subsection{Lie superalgebras}
A \textit{Lie superalgebra} is a superspace $\Lf=\Lf_0 \oplus \Lf_1$ equipped with 
a bilinear operation $[\;,\;]:\Lf\times \Lf\rightarrow \Lf$ satisfying
\begin{enumerate}
\item[(SL1)] skewsymmetry: $[x, y]=-(-1)^{\bar{x}\bar{y}}\,[y, x],$
\item[(SL2)] super Jacoby identity :
$$(-1)^{\bar x \bar z}[[x,y],z]+(-1)^{\bar y \bar x}[[y,z],x]+(-1)^{\bar z \bar y}[[z,x],y]=0.$$
\end{enumerate}
\noindent
for all $x,y,z$ homogeneous elements in $ \Lf$.\\

\noindent
\textbf{Commutator.}
Given any  associative superalgebra $\Ab$,
a natural Lie superbracket on $\Ab$ is given by the commutator
$[\;,\;]$ defined by
$$
[A,B]=A B-(-1)^{\bar{A}\bar{B}}B A
$$
for homogeneous elements $A, B\in\Ab$ and extended by bilinearity on $\Ab\times\Ab$.\\

\noindent
\textbf{Representations of Lie superalgebras.}
A representation of a Lie superalgebra $(\Lf,[\;,\;])$ is a a superspace  $V=V_0\oplus V_1$ 
together with a linear map
$
\r: \Lf \rightarrow \End(V),
$
satisfying 
$$
\r([x,y])=[\r(x),\r(y)], 
$$
for all $x,y\in\Lf$.

\subsection{Jordan superalgebras}
A superalgebra $(\Jc=\Jc_0\oplus \Jc_1, \, \cdot\,)$ is a \textit{Jordan superalgebra} if the product satisfies
\begin{enumerate}
\item[(SJ1)] supercommutativity: $a\cdot b=(-1)^{\bar{a}\bar{b}}\,b\cdot a,$
\item[(SJ2)] super Jordan identity :
$$
\begin{array}{ll}
&(a\cdot b)\cdot(c\cdot d)+(-1)^{\bar{b}\bar{c}}\,(a\cdot c)\cdot(b\cdot d)+ (-1)^{(\bar{b}+\bar{c})\bar{d}}(a\cdot d)\cdot(b\cdot c)\\[8pt]
&\;=\;((a\cdot b)\cdot c)\cdot d +  (-1)^{(\bar{b}+\bar{c})\bar{d}+\bar{b}\bar{c}}((a\cdot d)\cdot c)\cdot b+(-1)^{(\bar{b}+\bar{c}+\bar{d})\bar{a}+\bar{c}\bar{d}}((b\cdot d)\cdot c)\cdot a,
\end{array}
$$
\end{enumerate}
\noindent
for all $a,b,c,d $ homogeneous elements in $ \Jc$.\\

\noindent
\textbf{Anticommutator.}
Given any  associative superalgebra $\Ab$,
a natural Jordan superbracket is given by the anticommutator $[\;,\;]_+$ defined by
$$
[A,B]_+=A B+(-1)^{\bar{A}\bar{B}}B A
$$
for homogeneous elements $A, B\in\Ab$ and extended by bilinearity on $\Ab\times\Ab$.\\

\noindent
\textbf{Representations of Jordan superalgebras.}
A representation of a Jordan superalgebra $(\Jc,\cdot\,)$ is a a superspace  $V=V_0\oplus V_1$ 
together with a linear map
$
\r: \Jc \rightarrow \End(V),
$
satisfying 
$$
\r(a\cdot b)=[\r(a),\r(b)]_+, 
$$
for all $a,b\in\Jc$. A faithful embedding of a Jordan algebra into an associative algebra equipped with the anticommutator is also called a \textit{specialization}.

\subsection{Lie antialgebras}
Lie antialgebras form a subclass of Jordan superalgebras in which the algebras satisfy
cubic identities (instead of the quartic identities defining Jordan algebras).
They were introduced in a geometric setting in \cite{Ovs}.
However the defining axioms of Lie antialgebras already appeared in \cite{Kap}, \cite{McC2}. 
Thanks to the "simplified" cubic identities, 
one can develop new objects and notions associated to these particular Jordan algebras
(a specific representation theory \cite{MG},\cite{LMG}, cohomology theory \cite{LO}). The most important object here will be the adjoint Lie superalgebra constructed in \cite{Ovs}, which is different from the one obtained by applying the Koecher-Kantor-Tits process.\\

\noindent
\textbf{Definition.}
A \textit{Lie antialgebra} is a superalgebra $\Ac=\Ac_0\oplus\Ac_1$
with a supercommutative product satisfying the following cubic identities:
\begin{enumerate}
\item[(LA0)] associativity of $\Ac_0$
$$x_1\cdot (x_2 \cdot x_3)=(x_1\cdot x_2)\cdot x_3,$$
for all $x_1, x_2, x_3 \in \Ac_0$,
\item[(LA1)] half-action
\begin{equation*}
x_1\cdot\left(x_2\cdot{}y\right)=
\half \left(x_1\cdot{}x_2\right)\cdot{}y,
\end{equation*}
for all $x_1,x_2\in\Ac_0$ and $y\in\Ac_1$, 

\item[(LA2)] Leibniz identity
\begin{equation*}\label{eq3}
x\cdot\left(y_1\cdot{}y_2\right)=
\left(x\cdot{}y_1\right)\cdot{}y_2+
y_1\cdot\left(x\cdot{}y_2\right),
\end{equation*}
for all $x\in\Ac_0$ and $y_1,y_2\in\Ac_1$, 

\item[(LA3)] odd Jacobi identity
\begin{equation*}\label{eq4}
y_1\cdot\left(y_2\cdot{}y_3\right)+
y_2\cdot\left(y_3\cdot{}y_1\right)+
y_3\cdot\left(y_1\cdot{}y_2\right)=0,
\end{equation*}
for all $y_1,y_2,y_3\in\Ac_1$.

\end{enumerate}

The fact that the axioms of Lie antialgebras imply those of Jordan superalgebras can be found in \cite{McC2} (see also \cite{LMG} for more details).

\medskip

\noindent
\textbf{Adjoint Lie superalgebra.}\label{ols}
Given a Lie antialgebra $\Ac$, the \textit{adjoint  Lie superalgebra} denoted by $\Of\lf\Sf(\Ac)$  is defined as follows.
As a vector space $\Of\lf\Sf(\Ac)=\Of\lf\Sf(\Ac)_0\oplus \Of\lf\Sf(\Ac)_1$, where
$$
\Of\lf\Sf(\Ac)_1:=\Ac_1, \quad \Of\lf\Sf(\Ac)_0:=\Ac_1 \otimes \Ac_1/ \Ic
$$
and $\Ic$ is the ideal generated by
$$\{a\otimes b-b\otimes a,\; ax\otimes b-a\otimes bx\,| \; a,b\in \Ac_1, x \in \Ac_0\}.$$
We denote by $a\odot b$ the image of $a\otimes b$ in $\Of\lf\Sf(\Ac)_0$. Therefore, we have the following useful relations in $\Of\lf\Sf(\Ac)_0$:
\begin{equation*}\label{relequiv}
\left\{
\begin{array}{rcl}
a\odot b& = &b\odot a,\\
ax\odot b& = &a\odot bx \;= \;b\odot ax\; =\;bx\odot a, \quad \;a,b\in \Ac_1, x \in \Ac_0 .\\
\end{array}
\right.
\end{equation*}

\noindent
The Lie superbracket on $\Of\lf\Sf(\Ac)$ is given by:
\begin{equation}\label{olsbrak}
\begin{array}{rcl}
[a,b]&=& a\odot b,\\[5pt]
[a\odot b, c]=-[c,a\odot b]&=&a(bc)+b(ac),\\[5pt]
[a\odot b, c\odot d]&=& 2\,a(bc)\odot d + 2\, b(ad)\odot c,
\end{array}
\end{equation}
where $a,b,c$ and $d$ are elements of $\Of\lf\Sf(\Ac)_1=\Ac_1$.

\medskip

\noindent
\textbf{Representations of Lie antialgebras.}
Since Lie antialgebras are particular Jordan superalgebras, 
we will be interested in particular Jordan representations.
We call \textit{LA-representation} of the Lie antialgebra $(\Ac,\cdot\,)$ any Jordan representation $(\r, V)$ satisfying the additional condition
$$
\r(a)\r(b)=\r(b)\r(a), \; \text{ for all even elements }\, a, b\in \Ac_0.
$$
An important feature of LA-representations is that they can be extended to representations of the adjoint Lie superalgebra, \cite{Ovs}. The converse is not true but some "good representations" of the Lie superalgebra give rise to LA-representations, see \cite{LMG} and also Theorem \ref{goodrep} below.

\begin{ex}\label{K1AK1}
(a) The first example of finite dimensional Lie antialgebra is the {\it tiny} Kaplansky superalgebra, often denoted\footnote{In the first version of \cite{Ovs} this algebra was denoted $\mathfrak{asl}_2$; this notation is used in \cite{MG}.}
$\Kc_3$. In this case the adjoint Lie superalgebra is the orthosymplectic algebra  $\mathfrak{osp}(1|2)$.

(b) An example of infinite dimensional Lie antialgebras, related to vector fields over the line, is the following algebra $\Ac\Kc(1)=\langle \e_n, n\in\Z\rangle\oplus\langle a_i, i\in\Z+\half\rangle$, satisfying
\begin{equation*}
\left\{
\begin{array}{rcl}
\e_n\cdot \e_m&=& \e_{n+m}\\[5pt]
\e_n\cdot a_i&=& \half a_{n+i}\\[5pt]
a_i\cdot a_j&=& \half (j-i) \e_{i+j}.\\
\end{array}
\right.
\end{equation*}
In this case the adjoint Lie superalgebra $\Of\lf\Sf(\Ac\Kc(1))$ is the Neveu-Schwarz superalgebra
$\Kf(1)=\langle L_n,n\in\Z\rangle\oplus\langle A_i, i\in\Z+\half\rangle$ in which
\begin{equation*}
\left\{
\begin{array}{lcl}
\left[
L_n,L_m
\right] &=&
\half \left(m-n\right)L_{n+m},\\[8pt]
\left[
L_n,A_j
\right] &=&\half
\left(i-\frac{n}{2}\right)A_{n+i},\\[8pt]
\left[
A_i,A_j
\right] &=&
\,L_{i+j}.
\end{array}
\right.
\end{equation*}
\end{ex}

\section{Geometric construction}\label{geom}
In this section, we define the superalgebras of Krichever-Novikov type associated to an arbitrary punctured Riemann surface and study their main properties. 
We stress on the case of the sphere with three punctures.

\subsection{Generalized Krichever-Novikov algebras}
Let $\Si$ be a compact Riemann surface of arbitrary genus~$g$, or equivalently, a smooth projective curve over $\C$.
Choose a set of $N$ distinct points $P=\{P_1,\ldots,P_N\}$, called \textit{punctures}, on $\Si$.
Denote $\Ab_{g,N}$ the associative algebra 
consisting of meromorphic functions 
on $\Si$ which are holomorphic outside the set of punctures with point-wise
multiplication.
The Krichever-Novikov algebra $\gf_{g,N}$
is the Lie algebra consisting of meromorphic vector fields on $\Si$
which are holomorphic outside of the set of punctures, with the usual Lie bracket of
vector fields expressed locally as
$$
[f,g]=\left[f(z)\ddz\;,\; g(z)\ddz\right]= \left(f(z)g'(z)-f'(z)g(z)\right) \ddz
$$

Both $\Ab_{g,N}$ and $\gf_{g,N}$ are infinite dimensional algebras. 
In the case of 2 punctures on the sphere the algebra $\gf_{g,N}$ is nothing but the Witt algebra.
The algebras, and their extensions, obtained in the case of 2 punctures in higher genus were introduced  and studied by 
Krichever and Novikov \cite{KN1}, \cite{KN2}, \cite{KN3}.

\subsection{Superalgebras of Krichever-Novikov type}
To the above geometric situation, one can associate a Lie superalgebra and a Jordan superalgebra (which is a Lie antialgebra).
Denote by $\F_{\l}$ the space of tensor densities of weight $\l$, $\l \in \Z\cup\half+\Z$
(in the sequel most of the time $\l$ will take the value $-1, -\half, 0$).
One has the following natural space identifications 
$$
\Ab_{g,N} \cong \F_0\;,\qquad \gf_{g,N}\cong \F_{-1}\; .
$$
The products of the algebras can be realized in the density modules.
More generally, consider the following bilinear operations, given in local coordinates:
$$
\begin{array}{lcll}
\bullet \;: &\F_{\l}\times \F_{\mu} &\longrightarrow &\F_{\l+\mu}\\[8pt]
&\left(f(z)(dz)^{\l},g(z)(dz)^{\mu}\right)&\mapsto &f(z)g(z)(dz)^{\l+\mu}
\end{array}
$$
and
$$
\begin{array}{lcll}
\{ \;,\;\} : &\F_{\l}\times \F_{\mu} &\longrightarrow &\F_{\l+\mu+1}\\[8pt]
&\left(f(z)(dz)^{\l},g(z)(dz)^{\mu}\right)&\mapsto &\big(\mu f'(z)g(z)-\l f(z)g'(z)\big)(dz)^{\l+\mu+1}.
\end{array}
$$
These operations endow the space $\oplus_\l \F_\l$ with a structure of Poisson algebra.

The algebras $\Ab_{g,N}$ and $\gf_{g,N}$ naturally act on $\F_{-\half}$. 
Furthermore, one can construct a structure of Lie superalgebra, resp. Jordan superalgebra, 
on the space $ \gf_{g,N}\oplus\F_{-\half}$, resp. $\Ab_{g,N}\oplus\F_{-\half}$.

\begin{defn}\label{lknakn}
(i) The space $\gf_{g,N}\oplus\F_{-\half}$ equipped with the bracket $[\;,\;]$ given in local coordinates by
\begin{equation}\label{LKN}
\begin{array}{rcl}
\left[f(z)(dz)^{-1}\;,\;g(z)(dz)^{-1}\right]&=&\left\{ f(z)(dz)^{-1}, g(z)(dz)^{-1}\right\}\\[8pt]
\left[f(z)(dz)^{-1}\;,\;\g(z)(dz)^{-\half}\right]&=&\left\{f(z)(dz)^{-1}, \g(z)(dz)^{-\half}\right\}\\[8pt]
\left[\vp(z)(dz)^{-\half}\;,\;\g(z)(dz)^{-\half}\right]&=&\half \,\vp(z)(dz)^{-\half}\bullet \g(z)(dz)^{-\half}
\end{array}
\end{equation}
is a Lie superalgebra. 
We call it Lie superalgebra of Krichever-Novikov type and denote $\Lf_{KN}$.

(ii) The space $\Ab_{g,N}\oplus\F_{-\half}$ 
equipped with the product $\circ $ given in local coordinates by
\begin{equation}\label{JKN}
\begin{array}{rcl}
f(z)\circ g(z)&=&f(z)\bullet  g(z)\\[8pt]
f(z)\circ \g(z)(dz)^{-\half}&=&\half\,f(z)\bullet \g(z)(dz)^{-\half}\\[8pt]
\vp(z)(dz)^{-\half}\circ \g(z)(dz)^{-\half}&=&\left\{\vp(z)(dz)^{-\half}, \g(z)(dz)^{-\half}\right\}
\end{array}
\end{equation}
is a Jordan superalgebra (which is also a Lie antialgebra). 
We call it Jordan superalgebra of Krichever-Novikov type and denote $\Jc_{KN}$.

\end{defn}

The fact that \eqref{LKN} defines a Lie superbracket is well known
(this can also be checked directly from the definitions).
The fact that $\Jc_{KN}$ is a Jordan superalgebra 
comes from a more general construction starting from an associative algebra $\Ab$ and a derivation $D$ on $\Ab$ (here $\Ab=\Ab_{g,N}$ and $D=(dz)^{-1}$ as an element of $\gf_{g,N}$), see \cite{McC2} or Section \ref{JsAD} for more details.
One can check by direct computation that \eqref{JKN} satisfies the axioms of Lie antialgebra.

\begin{rem}\label{unitalJKN}
Alternatively, a unital Jordan algebra (which is not a Lie antialgebra) can be defined by modifying the 
product in $\Jc_{KN}$ with
$$f(z)\circ \g(z)(dz)^{-\half}=\,f(z)\bullet \g(z)(dz)^{-\half},$$
in the second equation of \eqref{JKN}.
\end{rem}

\begin{ex}\label{2pts}
The first example is the case of two punctures on the sphere:
\begin{equation*}\label{Wittgeom}
\Si=\Pb^1(\C), \quad P=\{0,\infty\}.
\end{equation*}
For the algebra of meromorphic functions
one obtains
$\C[z, z^{-1}]$, the algebra of Laurent polynomials.
The vector field algebra is the famous Witt algebra
$\Wf$ generated by 
$$
L_n(z)=z^{n+1}\ddz,\ n\in\Z,
$$
satisfying:
$$
[L_n,L_m]=(m-n)\,L_{n+m}.
$$
In this case the associated Lie superalgebra and Jordan superalgebra defined in Definition~ \ref{lknakn} are
$$
\Lf_{KN}\simeq \Kf(1), \qquad \Jc_{KN}\simeq\Ac\Kc(1),
$$
where $\Kf(1)$  and $\Ac\Kc(1)$ are as in Example \ref{K1AK1}.

The relation between these two algebras was given in \cite{Ovs} in terms of contact vector fields on the supercircle.
\end{ex}

We come back to the general case and state our first result.
\begin{thm}\label{KNlink}
The Krichever-Novikov superalgebras are related by
\begin{eqnarray*}
(i)\quad
\Lf_{KN}&\cong& \Of\lf\Sf(\Jc_{KN}),\\
(ii) \quad
\Lf_{KN}&\cong& \Der(\Jc_{KN}).
\end{eqnarray*}
\end{thm}

\begin{proof}
(i) Since $\Jc_{KN}$ is a Lie antialgebra,
one can apply the construction described in Section \ref{ols}.
The isomorphism between the Lie superalgebras  $\Lf_{KN}$ and
$\Of\lf\Sf(\Jc_{KN})$ is given by
$$
\begin{array}{rcl}
f(z)(dz)^{-\half}&\mapsto& f(z)(dz)^{-\half}\\[8pt]
f(z)(dz)^{-1}&\mapsto& 2(dz)^{-\half}\odot f(z)(dz)^{-\half}.
\end{array}
$$
(ii)
The algebra $\Der(\Jc_{KN})$ is the Lie subalgebra of
$(\End(\Jc_{KN}),[\,,\,])$ such that
any element $D\in \Der(\Jc_{KN})$ can be written as
$$
D=D_0+D_1,
$$
where $D_0$ and $D_1$ are even and odd endomorphisms, respectively,
satisfying
$$
D_i(A\circ B)=D_i(A)\circ B+(-1)^{i\bar A}A\circ D_i(B)\;,
$$
for all homogeneous elements $A,B$ in $\Jc_{KN}$ (recall that $\circ$
is the product on $\Jc_{KN}$ that is defined using the operation $\bullet$ and $\{\,,\,\}$
according to the parity of the elements).

One can naturally embed $\Lf_{KN}$ into $\Der(\Jc_{KN})$. Indeed,
for any even element $f\in \Lf_{KN}$, i.e. $f\in \F_{-1}$, and any odd
element
$\vp\in \Lf_{KN}$, i.e. $\vp\in \F_{-\half}$, define endomorphisms of
$\Jc_{KN}$ by
$$
\left\{
\begin{array}{rcll}
R_f(a)&=& \{a, f\}, &\forall \;a\in \Ab_{g,N}\\
R_f(\om )&=& \{\om, f\}, &\forall \; \om\in \F_{-\half}\\
\end{array}
\right.
,
\quad
\left\{
\begin{array}{rcll}
R_\vp(a)&=& \half  a\bullet \vp, &\forall \; a\in \Ab_{g,N}\\
R_\vp(\om )&=& \{\om, \vp\}, &\forall \; \om\in \F_{-\half}\\
\end{array}
\right.
.
$$
One can easily see that $R_f$ and $R_\vp$ are elements of $\Der(\Jc_{KN})$
(this also can be deduced from a more general statement, Lemma 3.2 in
\cite{Ovs}).
Let us show that every element in $\Der(\Jc_{KN})$ is of this form.

Case (a): Consider an even derivation $D$ in $\Der(\Jc_{KN})$.
The restriction of $D$ to $ \Ab_{g,N}$ is an element of
$\Der(\Ab_{g,N})$.
It is well known that $\gf_{g,N}=\Der(\Ab_{g,N})$, through the natural
right action.
Thus, there exists $f\in \gf_{g,N}$ such that
$$
D(a)=\{a, f\}, \quad \forall\; a\in  \Ab_{g,N}.
$$
In the sequel, the computations are made using a local coordinate $z$, but
to simplify the notation we often drop off the variable.
Introduce
$$
\d(z)(dz)^{-\half}=D\left(1\, dz^{-\half}\right),
$$
and let us show that $\d(z)=\half f'(z)$.
Using the property of derivation, we can write for all $\vp$
$$
D\left(\left\{\vp\, dz^{-\half}\,,\, dz^{-\half}\right\}\right)=
\left\{D\left(\vp\, dz^{-\half}\right)\,,\, dz^{-\half}\right\}
+\left\{\vp\, dz^{-\half}\,,\, D\left(\, dz^{-\half}\right)\right\}.
$$
In the above equality,
\begin{eqnarray*}
\text{LHS}&=&\textstyle D\left(-\half \vp'\right)=\half \vp''f,\\[6pt]
\text{RHS}&=&\left\{D\left(\vp\right)\, dz^{-\half}+\vp
\d\, dz^{-\half}\,,\,\, dz^{-\half}\right\}
+\left\{\vp\, dz^{-\half}\,,\,\d\, dz^{-\half}\right\}\\[6pt]
&=&\left\{(-\vp'f+\vp \d)\, dz^{-\half}\, ,\, dz^{-\half}\right\}
+\left\{\vp\, dz^{-\half}\,,\,\d\, dz^{-\half}\right\}\\[6pt]
&=&\half \vp''f+\half \vp'f'-\vp'\d.
\end{eqnarray*}
Since the equality holds for all $\vp$, we deduce $\d(z)=\half f'(z)$.

Now, we compute for all $\om\, dz^{-\half}\in \F_{-\half}$,
\begin{eqnarray*}
D\left(\om\, dz^{-\half}\right)\;=\;D\left(2\om\circ \, dz^{-\half}\right)&=&
D\left(2\om\right)\circ \, dz^{-\half}+2\om \circ D\left(\, dz^{-\half}\right)\\[4pt]
&=&
\left\{\om , f\right\}\, dz^{-\half}+\om D\left(\, dz^{-\half}\right)\\[4pt]
&=&-\om'f\, dz^{-\half}+\half \om f'\, dz^{-\half}\\[4pt]
&=&\left\{\om\, dz^{-\half}, f(dz)^{-1}\right\}.
\end{eqnarray*}
We have proved in the case of even derivation that $D=R_f$.

Case (b): Consider an odd derivation $D$ in $\Der(\Jc_{KN})$.
Introduce
$$
\vp(z)\, dz^{-\half}:=D(2),
$$
and let us show that $D\left(dz^{-\half}\right)=\half \vp'(z)$.
Writing
\begin{eqnarray*}
D\left(dz^{-\half}\right)=
D\left(2\circ dz^{-\half} \right)
&=&
D(2)\circ dz^{-\half}+2\circ D(dz^{-\half})
\\[4pt]
&=&\left\{\vp dz^{-\half}, dz^{-\half}\right\}+2D(dz^{-\half})\\[4pt]
&=&-\half\vp'+2D(dz^{-\half}),
\end{eqnarray*}
we deduce  $D\left(dz^{-\half}\right)=\half \vp'(z)$.

Now, it is easy to compute
$$
D(a)=\half a\bullet \vp dz^{-\half}, \quad \forall \; a\in  \Ab_{g,N},
$$
and
$$
D(\om dz^{-\half})=\left\{dz^{-\half}, \vp dz^{-\half}\right\}, \quad \forall\; \om
\in\F_{-\half}.
$$
Consequently, one has $D=R_\vp$.
\end{proof}

\begin{rem}
In general, given a Lie antialgebra $\Ac$ one always has
an action, by right multiplication, of $\Of\lf\Sf(\Ac)$ on $\Ac$, i.e.
an inclusion $\Of\lf\Sf(\Ac)\hookrightarrow \Der(\Ac)$, but it is not
necessarily an isomorphism, \cite{Ovs}.
Isomorphisms were established in the cases $\Ac=\Kc_3$ and $\Ac=\Ac\Kc(1)$.
Theorem \ref{KNlink} enlarges the class of Lie antialgebras for which one
has the identification $\Of\lf\Sf(\Ac)\cong\Der(\Ac)$.
\end{rem}

\subsection{Representations}\label{RepSec}
An important result in the representation theory of Lie antialgebras \cite{Ovs}
is the fact that any LA-representation of a Lie antialgebra $\Ac$
generates a representation of the Lie superalgebra $\Of\lf\Sf(\Ac)$.
The converse is in general not true. 
However, it is surprizing that in some cases the action of the odd elements
of $\Of\lf\Sf(\Ac)$ considered with the antcommutator $[\, ,\,]_+$ generate a representation of $\Ac$.
This feature is developped in this section.

Consider the vector superspace $\Vc_\l=\F_\l\oplus\F_{\l+\half}$, where $\l\in \Z\cup \half+\Z$. 
The elements of  $\Vc_\l$ belonging to $\F_\l$, resp. $\F_{\l+\half}$, are considered as even, resp. odd.
We give natural actions of the algebras $\Lf_{KN}$ and $\Jc_{KN}$ on $\Vc_\l$.

Define the linear map 
$\widetilde{\r}: \Lf_{KN}=\gf_{g,N}\oplus\F_{-\half}\rightarrow \End(\Vc_\l)$ by
\begin{equation}
\widetilde{\r}_{f \choose\vp}
\left(\begin{array}{c}v\\[8pt]
 \om\end{array}\right)=
\left(
\begin{array}{lcl}
\{f, v\}&+&\half \vp \bullet \om\\[8pt]
\{f, \om\}& +& \{\vp, v\}
\end{array}
\right)
\end{equation}
where $f\in\gf_{g,N},\; \vp\in\F_{-\half}, \;v\in\F_\l,\; \om \in\F_{\l+\half}$, and 
define the linear map $\r: \Jc_{KN}=\Ab_{g,N}\oplus\F_{-\half}\rightarrow \End(\Vc_\l)$ by
\begin{equation}
\r_{f \choose\vp}
\left(\begin{array}{c}v\\[8pt]
\om\end{array}\right)=
\left(
\begin{array}{ccc}
\l f\bullet v&+&\half \vp \bullet \om\\[8pt]
(\half-\l)f\bullet \om& +& \{\vp , v\}
\end{array}
\right)
\end{equation}
where $f\in\Ab_{g,N},\; \vp\in\F_{-\half}, \;v\in\F_\l,\; \om \in\F_{\l+\half}$.

Note that one has
$$
\widetilde{\r}_{\vert \F_{-\half}}=\r_{\vert \F_{-\half}}.
$$

\begin{thm}\label{goodrep}
(i) The map $\widetilde{\r}$ is a faithful representation of the Krichever-Novikov Lie superalgebra 
$\Lf_{KN}$ for any value of $\l$,

(ii) The map $\r$ is a faithful (LA-)representation of the Krichever-Novikov Jordan superalgebra $\Jc_{KN}$ if and only if
$\l=0$ or $\half$.
\end{thm}

\begin{proof}
Point (i) is a classical fact.
Point (ii) can be established by direct computations.
Indeed, one can check that the identities
$$
[\r(\vp),\r(\g)]_+=\r(\vp\circ \g),\qquad [\r(f),\r(\vp)]_+=\r(f\circ \vp),
$$
are always satisfied for any odd elements $\vp,\g$ and even element $f$ in $\Jc_{KN}$.
Whereas the identity involving two even elements 
$$
[\r(f),\r(g)]_+=\r(f\circ g),
$$
is satisfied if and only if $\l=0$ or $\half$.
\end{proof}

\begin{rem}
In other words, Theorem 2 implies that the actions of odd elements of $\Lf_{KN}$ on $\Vc_\l$
generate a Jordan subalgebra of $\left(\End(\Vc_\l), [\;,\;]_+\right)$, for $\l=0,\half$, isomorphic to $\Jc_{KN}$.
\end{rem}

\subsection{The case of three punctures  on the sphere}\label{SecThmSubalg}
Consider the three point situation in genus 0:
\begin{equation*}\label{KNgeom}
\Si=\Pb^1(\C), \quad P=\{\a,-\a, \infty\},
\end{equation*}
where $\a\in\C\setminus\{0\}$. This case has been studied in \cite{Schli}, \cite{FS}, \cite{FS2}.

Note that the moduli space $\mathcal{M}_{0,3}$ is trivial so that the constructions do not depend on the choice of $\a$.

The corresponding function algebra $\Ab_{0,3}$ has basis  $\{G_n,\; n\in \Z\}$, where the functions are locally defined by
$$
G_{2k}(z)= (z-\a)^k(z+\a)^k, \qquad G_{2k+1}(z)= z(z-\a)^{k}(z+\a)^{k},
$$
and satisfying
\begin{equation}\label{FKN}
G_n \;G_m=\left\{
\begin{array}{ll}G_{n+m}+\a^2G_{n+m-2},&n, m \text{ odd,}\\[5pt]
G_{n+m},& \text{ otherwise. }\\
\end{array}
\right.
\end{equation}
The algebra of vector fields $\gf_{0,3}$ has basis $\{V_n\}_{n\in \Z}$, where
$$
V_{2k}(z)= z(z-\a)^k(z+\a)^k\ddz, \qquad V_{2k+1}(z)= (z-\a)^{k+1}(z+\a)^{k+1}\ddz,
$$
satisfying the relation
\begin{equation}\label{KN}
[V_n\; , \;V_m]=\left\{
\begin{array}{ll}(m-n)V_{n+m},&n,m \text{ odd, }\\[5pt]
(m-n)V_{n+m}+(m-n-1)\a^2V_{n+m-2},&n \text{ odd},\; m \text{ even,}\\[5pt]
(m-n)(V_{n+m}+\a^2V_{n+m-2}),&n,m \text{ even.}\\[5pt]
\end{array}
\right.
\end{equation}

The next proposition gives the description in terms of generators and relations, of the superalgebras of Krichever-Novikov type obtained in the particular case of three punctured sphere.

\begin{prop}\label{JKNLKN}
(i) The Lie superalgebra of Krichever-Novikov type, $\Lf_{0,3}=\gf_{0,3} \oplus \F_{-\half}$, has even basis vectors $V_n$, $n\in \Z$,
and odd basis vectors $\vp_i$, $i\in \Z+\half$, satisfying the relations \eqref{KN} and
\begin{equation*}
\begin{array}{lcl}
[V_n\; , \;\vp_i]&=&\left\{\begin{array}{lll}
(i-\frac{n}{2})\vp_{n+i},&n \text{ odd},&i-\half \text{ odd},\\[5pt]
(i-\frac{n}{2})\vp_{n+i}+(i-\frac{n}{2}-1)\a^2\vp_{n+i-2},&n \text{ odd},&i-\half \text{ even},\\[5pt]
(i-\frac{n}{2})\vp_{n+i}+(i-\frac{n}{2}+\half)\a^2\vp_{n+i-2},&n \text{ even},&i-\half \text{ odd},\\[5pt]
(i-\frac{n}{2})\vp_{n+i}+(i-\frac{n}{2}-\half)\a^2\vp_{n+i-2},&n \text{ even},&i-\half \text{ even},\\[8pt]
\end{array}\right.\\[20pt]
[\vp_i\; , \;\vp_j]&=&\left\{\begin{array}{ll}
V_{i+j}+\a^2V_{i+j-2},&i-\half\text{ even }, \,j-\half \text{ even,}\\[5pt]
V_{i+j}, & \text{ otherwise}.\\[5pt]
\end{array}
\right.
\end{array}
\end{equation*}
(ii) The Jordan superalgebra of Krichever-Novikov type, $\Jc_{0,3}=\Ab_{0,3} \oplus \F_{-\half}$ has even basis vectors $G_n$, $n\in \Z$,
and odd basis vectors $\vp_i$, $i\in \Z+\half$, satisfying the relations \eqref{FKN} and
\begin{equation*}
\begin{array}{lcl}
G_n\; \circ \;\vp_i&=&\left\{\begin{array}{lll}
\half \vp_{n+i},&n \text{ even or}&i-\half \text{ odd},\\[5pt]
\half(\vp_{n+i}+\a^2\vp_{n+i-2}),&n \text{ odd and}&i-\half \text{ even},\\[8pt]
\end{array}\right.\\[20pt]
\vp_i\; \circ  \;\vp_j&=&\left\{\begin{array}{lll}
(j-i)G_{i+j},&i-\half \text{ odd},& j-\half \text{ odd},\\[5pt]
(j-i)G_{i+j}+(j-i+1)\a^2G_{i+j-2},&i-\half \text{ even},& j-\half \text{ odd},\\[5pt]
(j-i)(G_{i+j}+\a^2G_{i+j-2}),&i-\half \text{ even},& j-\half \text{ even}.\\[5pt]
\end{array}
\right.
\end{array}
\end{equation*}
\end{prop}

\begin{proof}
This can be established by direct computations using  the following notation
$$\vp_{2k+\half}=\sqrt{2}z(z-\a)^k(z+\a)^k(dz)^{-\half}\;\;,\quad 
\vp_{2k-\half}=\sqrt{2}(z-\a)^{k}(z+\a)^{k}(dz)^{-\half}.
$$
to express locally the elements of the density space $\F_{-\half}$.
\end{proof}

\subsection{Embeddings $\Kf(1)\subset \Lf_{0,3}$ and $\Ac\Kc(1)\subset \Jc_{0,3}$}
One can naturally recover the algebras obtained in the case of two punctures inside the ones obtained from three punctures. 
This corresponds to restriction of the set of labeling integers in the presentation of 
 $\Lf_{0,3}$ and $\Jc_{0,3}$
to nonpositive integers, so that one only keeps the functions and vector fields which are holomorphic at infinity.

\begin{prop}\label{isoWitt}
(i) The subalgebra $\Lf_{0,3}^-:=\langle V_n, n\leq 0;\; \vp_i, i\leq \half \rangle$ of $\Lf_{0,3}$ 
is isomorphic to $\Kf(1)$.

(ii) The subalgebra $\Jc_{0,3}^-:=\langle G_n, n\leq 0;\; \vp_i, i\leq \half \rangle$ of $\Jc_{0,3}$ 
is isomorphic to $\Ac\Kc(1)$.
\end{prop}

\begin{proof}
Point (i) and (ii) can be viewed geometrically using the following change in coordinates
$$\om=\frac{z-\a}{z+\a}.$$
Equivalently, the isomorphisms can be established using direct identification between the generators, 
like the following for case (ii):
$$
\e_{-1}=G_0+2\a G_{-1}+2\a^2 G_{-2},
\;
\e_{1}=G_0-2\a G_{-1}+2\a^2 G_{-2},
\;
 a_{-\half}=\textstyle\frac{1}{2\sqrt{\a}}( \vp_{\half}+\a \vp_{-\half}).$$
 \end{proof}

\section{Algebraic construction}\label{algeb}

In this section, we recover the superalgebras of Krichever-Novikov type described in
Section \ref{SecThmSubalg} in a purely algebraic way. 
It turns out that the construction is related to that of \cite{Zhe1}, \cite{Zhe2}.

\subsection{Doubling process}\label{JsAD}
We consider Jordan superalgebras of infinite dimension
which can be obtained using the following algebraic construction.
Let $A$ be a commutative associative complex algebra with unit
and $D$ be a derivation on $A$.
Consider the space  $\Jc_\s(A,D)=A\oplus \eta A$, where $\eta A$ is an isomorphic copy of $A$ considered
as an odd component, and $\s =1$ or $\half$ is a scalar parameter, together with the following supercommutative product:
\begin{equation}\label{JAD}
\left\{
\begin{array}{ccl}
a\circ b &=& ab\\[5pt]
a\circ \eta{}b&=& \s\, \eta(ab)\\[5pt]
\eta{}a \circ \eta{}b&=& aD(b)-D(a)b,
\end{array}
\right.
\end{equation}
for all $a,b\in A$.
This construction as well as various generalizations can be found in
\cite{Kant}, \cite{KMcC}, \cite{McC1}, \cite{McC2}, \cite{CaKa}.

The algebra $\Jc_1(A,D)$ is called \textit{vector type} Jordan superalgebra \cite{KMcC}, \cite{McC1}.
The algebra $\Jc_{\frac{1}{2}}(A,D)$ is called \textit{full derivation} Jordan superalgebra \cite{McC2}.
It is known that these algebras are simple iff $A$ has no non-trivial $D$-invariant ideals,  \cite{McC1}, \cite{McC2}.

The algebras $\Jc_1(A,D)$ and $\Jc_{\frac{1}{2}}(A,D)$  are not isomorphic. 
Indeed, the first one is unital whereas the second one is not
(it is half-unital).
We will show, Theorem~\ref{ThmRep}, that such algebras can be obtained from the representation of the same Lie superalgebra.

Direct computations lead to the following.
\begin{prop}
The algebra $(\Jc_{\frac{1}{2}}(A,D), \circ )$ is a Lie antialgebra.
\end{prop}
One can therefore associate a Lie superalgebra to $(\Jc_{\frac{1}{2}}(A,D), \circ )$ using the construction of Section \ref{ols}. 
Denote $\Lf(A,D)$ the Lie superalgebra $\Of\lf\Sf(\Jc_{\frac{1}{2}}(A,D))$.
In this context, the construction $\Lf(A,D)$ can be simplified and expressed in terms of a doubling process as well.

\begin{prop}
The algebra $\Lf(A,D)$ is isomorphic to $A\oplus \eta{}A$ equipped with the following
skewsymmetric superbracket
\begin{equation}\label{LAD}
\left\{
\begin{array}{ccl}
[a, b] &=&aD(b)-D(a)b\\[5pt]
[a, \eta{}b ]&=&\eta\left(aD(b)-\half D(a)b\right)\\[5pt]
[\eta{}a,\eta{}b]&=&ab,
\end{array}
\right.
\end{equation}
for all $a,b\in A$.
\end{prop}

\begin{proof}
Any even element in $\Lf(A,D)$ can be identified with an element of $A$ as follows
$$
\eta{}a\odot \eta{}b \equiv ab.
$$
Since $A$ is unital, the above identification does not depend on the representative
$\eta{}a\odot \eta{}b $.
Through this identification the bracket on $\Lf(A,D)$ given in \eqref{olsbrak} becomes as in \eqref{LAD}.
\end{proof}

\begin{ex}
The following choice
$$
A=\C[x] ,
\quad
D= \partial_x,
$$
leads in the case $\s=1$ to the well known Jordan superalgebras of vector fields on the line over $\C$, \cite{MeZ} and in the case $\s=\half$ to 
the Kaplansky-McCrimmon polynomial superalgebra, \cite{Kap}, \cite{McC2}.
The variant considering $A=\C[x,x^{-1}]$ would lead exactly to the algebra $\Ac\Kc(1)$
(given in Examples \ref{K1AK1} and \ref{2pts})

\end{ex}

\subsection{Main example}\label{ConsS2}
We apply the doubling process with the following choices
$$
A=\C[x,y^{\pm 1}]/(x^2-\t y^{2p}-1), 
\quad
D=x\partial_y+p\t y^{2p-1}\partial_x.
$$
where $\t\in \C^*$ and $p\in \Z^*$ are parameters.

We use the notation
$\jtp:=(\Jc_\s(A,D), \circ )$ when $A$ and $D$ are as above.
Constructions in \cite{Zhe1}, \cite{Zhe2} are based on this type of algebras.

\begin{prop}
The algebras $\jtp$ are simple.
\end{prop}

\begin{proof}
It is equivalent to show that the algebra $A$ has no non-trivial $D$-invariant ideals.
The proof given in \cite{Zhe2} in a particular case
can be easily adapted for arbitrary values of $\theta$ and $p$.
We sketch the proof here for the sake of completeness.

Assume $I$ is a non-zero $D$-invariant ideal of $A$.
Choose any element $f(y)+xg(y)$ in $I$, where $f,g\in \C[y^{\pm 1}]$. One has 
$$
f(y)^2-(1+\t y^{2p})g(y)^2=\big(f(y)+xg(y)\big)\big(f(y)-xg(y)\big)\in I.
$$
Therefore, we obtain that $I$ contains an element $h(y)$ of $\C[y^{\pm 1}]$. 
Multiplying by $y^{m}$ for some convenient $m \in \N$, 
we can assume that $h(y)$ belongs to $\C[y]$.

Now, we can prove by induction that the elements $x^{2k-1}h^{(k)}(y)$, 
where $h^{(k)}$ is the $k$-th derivative of $h$ with respect to $y$,
all belong to $I$. 
Indeed, writing that $D(h(y))=xh'(y)$ belongs to $I$ gives the property for $k=1$.
The induction is then based on the following equality: 
$$
D(x^{2k}h^{(k)}(y))=x^{2k+1}h^{(k+1)}(y)+2kp\t y^{2p-1}x^{2k-1}h^{(k)}(y).
$$
Consequently, we obtain that $I$ contains an element $x^m$, for a suitable $m\in \N$.
The following computation
$$
yD(x^m)=pmx^{m+1}-pmx^{m-1},
$$
implies that $x^{m-1}$ also belongs to $I$. 
By induction, this yields to $1$ belongs to $I$, and therefore $I$ is equal to $A$ itself.
\end{proof}

A presentation by generators and relations of the algebra $\jtp$ is the following.
$$\jtp=\langle x_n,\,y_n,\, a_i,\, b_i,\; n\in \Z,\;i\in \textstyle \half+\Z\rangle\;:$$
\begin{equation}
\left\{
\begin{array}{ccl}
x_n\; x_m&=&x_{n+m}\\[5pt]
x_n\;y_m&=&y_{n+m}\\[5pt]
y_n\;y_m&=&x_{n+m}+\t x_{n+m+2p}\\[5pt]
x_n\;a_j&=&\s a_{n+j}\\[5pt]
x_n\;b_j&=&\s b_{n+j}\\[5pt]
y_n\;a_j&=&\s b_{n+j}\\[5pt]
y_n\;b_j&=&\s (a_{n+j}+\t a_{n+j+2p})\\[5pt]
a_i\;a_j&=&(j-i)y_{i+j}\\[5pt]
a_i\;b_j&=&(j-i)x_{i+j}+\t (j-i+p)x_{i+j+2p}\\[5pt]
b_i\;b_j&=&(j-i)(y_{i+j}+\t y_{i+j+2p})\\[5pt]
\end{array}
\right.
\end{equation}

This presentation is obtained from the construction \eqref{JAD} using the notation 
$$
\textstyle x_n=y^n,\; y_n=xy^n, \;a_{n-\frac{1}{2}}=\eta{}y^{n},\;
b_{n-\frac{1}{2}}= \eta(xy^{n}).
$$

The Lie superalgebra $\Lf(\t,p)=\Of\lf\Sf(\Jc_{\half}(\t,p))$  is 
described as follows. 
$$\Lf(\t,p)=\langle L_n, H_n, A_i, B_i, \; n\in \Z,\;i\in \textstyle \half+\Z \rangle $$ 

\begin{equation*}
\left\{
\begin{array}{ccl}
[A_i\; , \;A_j]&=&L_{i+j}\\[5pt]
[B_i\; , \;B_j]&=&L_{i+j}+\t L_{i+j+2p}\\[5pt]
[A_i\; , \;B_j]&=&H_{i+j}\\[5pt]
[L_n\; , \;A_i]&=&(i-\frac{n}{2})B_{n+i}\\[5pt]
[L_n\; , \;B_i]&=&(i-\frac{n}{2})A_{n+i}+\t (i-\frac{n}{2}+p)A_{n+i+2p}\\[5pt]
[H_n\; , \;A_i]&=&(i-\frac{n}{2})A_{n+i}+\t (i-\frac{n}{2}-\frac{p}{2})A_{n+i+2p}\\[5pt]
[H_n\; , \;B_i]&=&(i-\frac{n}{2})B_{n+i}+\t (i-\frac{n}{2}+\frac{p}{2})B_{n+i+2p}\\[5pt]
[L_n\; , \;L_m]&=&(m-n)H_{n+m}\\[5pt]
[L_n\; , \;H_m]&=&(m-n)L_{n+m}+\t (m-n+p)L_{n+m+2p}\\[5pt]
[H_n\; , \;H_m]&=&(m-n)(H_{n+m}+\t H_{n+m+2p})\\[5pt]
\end{array}
\right.
\end{equation*}

The particular values $\t=\a^2$, $p=-1$ lead to the algebras described in Proposition \ref{JKNLKN}.
One immediately deduces the following

\begin{prop}\label{ThmSubalg} 
(i) The Lie superalgebra $\Lf_{0,3}$ is isomorphic to the subalgebra of $\Lf(\a^2,-1)$
generated by $\{H_{2k}, L_{2k+1},A_{2k+\half}, B_{2k-\half};\; k\in \Z\}$,

(ii) The Lie antialgebra $\Jc_{0,3}$ is isomorphic to the subalgebra of $\Jc_{\half}(\a^2,-1)$
generated by $\{x_{2k}, y_{2k+1},a_{2k+\half}, b_{2k-\half};\; k\in\Z\}$.
\end{prop}

\subsection{Interesting subalgebras}\label{newtype}
In \cite{Zhe1}, \cite{Zhe2}, the author constructs infinite dimensional Jordan superalgebras of ``new type'', in the sense that they are not isomorphic to an algebra of type $\Jc_\s(A,D)$ nor to Cheng-Kac superalgebras. The superalgebras in \cite{Zhe1}, \cite{Zhe2} are subalgebras of $\Jc_1(-1,p)$ for $p=1$ and $p=2$,
and are considered over an arbitrary ground field of characteristic zero (not necessarily $\C$).
Let us introduce the following notation
$$\Jc_\s(\t,p)^+=\langle x_{2k}, y_{2k+1},a_{2k+\half}, b_{2k-\half};\; k\in\N\rangle. $$

\begin{prop} \cite{Zhe1} \cite{Zhe2}
(i) If $-1$  is not a square in the ground field,
 then $\Jc_1(-1,1)^+$ is a simple Jordan superalgebra of ``new type'',
 
(ii) $\Jc_1(-1,2)^+$ is always a simple Jordan superalgebra of ``new type''.
\end{prop}

Similar statements hold for the half-unital algebras $\Jc_{\half}(\t,p)^+$.
The ``new type'' algebras $\Jc_{\half}(\t,1)^+$ can not be achieved in the geometric setting of Krichever-Novikov, due to solutions $\a^2=\t$ in $\C$, see Proposition \ref{isoWitt}.
The ``new type'' algebras $\Jc_{\half}(\t,2)^+$ can be realized geometrically considering the Krichever-Novikov algebras coming from a torus with two punctures, see Section \ref{torus}.

\subsection{From Lie representations to Jordan representations}\label{SecThmKN}
A remarkable property is that both Jordan algebras $\Jc_{\half}(A,D)$ and $\Jc_{1}(A,D)$ can be realized using the density modules of the Lie superalgebra $\Lf(A,D)$.
Here we describe explicitly this property for the algebras $\jtp$. 
The representations correspond to the ones defined geometrically in Section \ref{RepSec}.

Consider the infinite dimensional vector superspace $\Vc_\l$, with basis 
$\{f_m,g_m, \phi_j, \g_j\}$, $m\in \Z,\;j\in \half+\Z$, where $f_m$, $g_m$ are even elements and $\phi_j$, $\g_j$ odd elements. Define the following odd operators $A_i$, $B_i$ on $\Vc_\l$
\begin{equation}\label{actionL}
\left.
\begin{array}{ccl}
A_i\;\cdot \phi_j&=&f_{i+j}\\[5pt]
A_i\;\cdot \g_j&=&g_{i+j}\\[5pt]
A_i\;\cdot f_m&=&(\frac{m}{2}+\l i)\g_{m+i}\\[5pt]
A_i\;\cdot g_m&=&(\frac{m}{2}+\l i)\phi_{m+i}+\t (\frac{m}{2}+\l i +\frac{p}{2})\phi_{m+i+2p}\\[15pt]
B_i\;\cdot \phi_j&=&g_{i+j}\\[5pt]
B_i\;\cdot \g_j&=&f_{i+j}+\t f_{i+j+2p}\\[5pt]
B_i\;\cdot f_m&=&(\frac{m}{2}+\l i)\phi_{m+i}+\t (\frac{m}{2}+\l i +\l p)\phi_{m+i+2p}\\[5pt]
B_i\;\cdot g_m&=&(\frac{m}{2}+\l i)\g_{m+i}+\t (\frac{m}{2}+\l i +(\frac{1}{2}+\l)p)\g_{m+i+2p}\\
\end{array}
\right.
\end{equation}
and the following even operators $L_n$, $H_n$
\begin{equation*}
\left.
\begin{array}{ccl}
L_n\;\cdot \phi_j&=&\big( (\half +\l)n+j\big)\g_{n+j}\\[8pt]
L_n\;\cdot \g_j&=&\big( (\half +\l)n+j\big)\phi_{n+j}+ \t \big( (\half +\l)n+j+p\big)\phi_{n+j+2p}\\[8pt]
L_n\;\cdot f_m&=&(m+\l n)g_{m+n}\\[8pt]
L_n\;\cdot g_m&=&(m+\l n)f_{m+n}+\t (m+\l n+p)f_{m+n+2p}\\[15pt]
H_n\;\cdot \phi_j&=&\big( (\half +\l)n+j\big)\phi_{n+j}+\t\big( (\half +\l)(n+p)+j \big)\phi_{n+j+2p}\\[8pt]
H_n\;\cdot \g_j&=&\big( (\half +\l)n+j\big)\g_{n+j}+\t \big( (\half +\l)n+j+(\frac{3}{2}+\l)p\big)\g_{n+j+2p}\\[8pt]
H_n\;\cdot f_m&=&(m+\l n)f_{m+n}+\t(m+\l n+\l p)f_{m+n+2p}\\[8pt]
H_n\;\cdot g_m&=&(m+\l n)g_{m+n}+\t(m+\l n+(\l+1)p)g_{m+n+2p}.\\
\end{array}
\right.
\end{equation*}

\begin{prop}
The system \eqref{actionL} defines a representation  $ \Lf(\t,p) \rightarrow \End(\Vc_\l)$ of Lie superalgebras.
\end{prop}

\begin{proof}
This can be established by direct computations or deduced from Theorem~\ref{goodrep}.
\end{proof}

Define even endomorphisms $X_n$ and $Y_n$  of $V_\l$ by
\begin{equation*}\label{actionA}
\left.
\begin{array}{ccl}
X_n\;\cdot \phi_j&=&(\half -\l)\phi_{n+j}\\[8pt]
X_n\;\cdot \g_j&=&(\half -\l)\g_{n+j}\\[8pt]
X_n\;\cdot f_m&=&\l f_{m+n}\\[8pt]
X_n\;\cdot g_m&=&\l g_{m+n}\\
\end{array}
\right.
\qquad
\left.
\begin{array}{ccl}
Y_n\;\cdot \phi_j&=&(\half -\l)\g_{n+j}\\[8pt]
Y_n\;\cdot \g_j&=&(\half -\l)(\phi_{n+j}+\t \phi_{n+j+2p})\\[8pt]
Y_n\;\cdot f_m&=&\l g_{m+n}\\[8pt]
Y_n\;\cdot g_m&=&\l (f_{m+n}+\t f_{m+n+2p}).\\
\end{array}
\right.
\end{equation*}

\begin{thm}\label{ThmRep}
The subspace $\Sc:=\langle X_n, Y_n, A_i,B_i\rangle$ of endomorphisms  of $\Vc_\l$ defined in \eqref{actionL} and \eqref{actionA}
is a Jordan subalgebra of $\big(\End(\Vc_\l), [\,,\,]_+\big)$
if and only if 
$$\l=0, \textstyle\frac{1}{4}, \half.$$
Furthermore, one has the following isomorphisms of Jordan algebra
 \begin{enumerate}
 \item[\textbullet] $\Sc \simeq \Jc_{\half}(\t,p)$, if $\l=0, \half$,
 \item[\textbullet] $\Sc \simeq \Jc_{1}(\t,p)$,  if $\l=\frac{1}{4}$.
 \end{enumerate}
\end{thm}

\begin{proof}
By direct computation, one checks that the following holds true for any values of the parameter $\l$,
\begin{equation*}
\left\{
\begin{array}{ccl}
[X_n\;,\;A_j]_+&=&\half A_{n+j}\\[5pt]
[X_n\;,\;B_j]_+&=&\half B_{n+j}\\[5pt]
[Y_n\;,\;A_j]_+&=&\half B_{n+j}\\[5pt]
[Y_n\;,\;B_j]_+&=&\half (A_{n+j}+\t A_{n+j+2p})\\[5pt]
[A_i\;, \;A_j]_+&=&(j-i)Y_{i+j}\\[5pt]
[A_i\;, \;B_j]_+&=&(j-i)X_{i+j}+\t (j-i+p)X_{i+j+2p}\\[5pt]
[B_i\;, \;B_j]_+&=&(j-i)(Y_{i+j}+\t Y_{i+j+2p})\\[5pt]
\end{array}
\right.
\end{equation*}
In general, for arbitrary value of $\l$,
the Jordan bracket between even operators of $\Sc$ is not an element of $\Sc$. 
One has

\begin{equation*}\label{}
\left.
\begin{array}{ccl}
[X_n\;,\;X_m]_+\;\cdot \phi_j&=&2(\half -\l)^2\phi_{n+m+j}\;,\\[8pt]
[X_n\;,\;X_m]_+\;\cdot \g_j&=&2(\half -\l)^2\g_{n+m+j}\;,\\[8pt]
[X_n\;,\;X_m]_+\;\cdot f_k&=&2\l^2 f_{m+n+k}\;,\\[8pt]
[X_n\;,\;X_m]_+\;\cdot g_k&=&2\l^2 g_{m+n+k}\;,\\[15pt]
\end{array}
\right.
\quad
\left.
\begin{array}{ccl}
 X_{n+m}\;\cdot \phi_j&=&(\half -\l)\phi_{n+m+j}\;,\\[8pt]
 X_{n+m}\;\cdot \g_j&=&(\half -\l)\g_{n+m+j}\;,\\[8pt]
 X_{n+m}\;\cdot f_k&=&\l f_{m+n+k}\;,\\[8pt]
 X_{n+m}\;\cdot g_k&=&\l g_{m+n+k}\;,\\[15pt]
\end{array}
\right.
\end{equation*}
so that 
\begin{equation*}\label{}
\left.
\begin{array}{ccl}
[X_n\;,\;X_m]_+=\mu X_{n+m} 
\quad 
&\Longleftrightarrow&
\left\{
\begin{array}{ccl}
2(\half -\l)^2&=& \mu (\half -\l)\\
2\l^2&=& \mu \l
\end{array}
\right.\\[18pt]
\quad 
& \Longleftrightarrow&
\l=0, \mu=1 \text{ or } \l=\half, \mu=1 \text{ or } \l=\frac{1}{4}, \mu=\half.
\end{array}
\right.
\end{equation*}

Therefore, for $\l=0,\half$ and $\l=\frac{1}{4}$, one obtains respectively the following additional relations
\begin{equation*}
\left\{
\begin{array}{ccl}
[X_n\;,\;X_m]_+&=& X_{n+m}\\[5pt]
[X_n\;,\;Y_m]_+&=& Y_{n+m}\\[5pt]
[Y_n\;,\;Y_m]_+&=& X_{n+m}\\[5pt]
\end{array}
\qquad
,\qquad 
\right.
\left\{
\begin{array}{ccl}
[X_n\;,\;X_m]_+&=&\half X_{n+m}\\[5pt]
[X_n\;,\;Y_m]_+&=&\half Y_{n+m}\\[5pt]
[Y_n\;,\;Y_m]_+&=&\half X_{n+m}.\\[5pt]
\end{array}
\right.
\end{equation*}
In the case  $\l=0,\half$ we immediately obtain a specialization of $\Jc_{\half}(\t,p)$.
In the case $\l=\frac{1}{4}$ we obtain a specialization of $\Jc_{1}(\t,p)$ using the rescaling 
$X'_n=2X_n$, $Y'_n=2Y_n$, $A'_i=\sqrt{2}A_i$, $B'_i=\sqrt{2}B_i$.
\end{proof}

\begin{rem}
The cases $\l=0, \half$ correspond to the geometric situation described in Theorem \ref{goodrep}.
The case $\l=\quart$ does not appear in Theorem \ref{goodrep} as this value does not make sense for meromorphic tensor densities.
However, the unital algebra $\Jc_1(A,D)$ can be defined 
geometrically, see Remark \ref{unitalJKN}.
\end{rem}
\begin{rem}
The analog of Theorem \ref{ThmRep}, cases $\l=0, \half$,  was established 
 in \cite{LMG}. 
 for the algebras $\Ac\Kc(1)$ and $\Kf(1)$, the extra case $\l=\frac{1}{4}$ was not mentioned
 but can be easily achieved using formulas in \textit{loc.cit}.
 \end{rem}

\subsection{Krichever-Novikov superalgebras on the torus}\label{torus}
In \cite{Schli2}, \cite{FS}, \cite{FS2}, the case of two punctures on a surface of genus one is also studied.
The Krichever-Novikov superalgebras $\Lf_{KN}$ and $\Jc_{KN}$ associated to this case can be described algebraically as follows.
Consider 
\begin{equation*}
A=\C[x,y^{\pm 1}]/(x^2+\t_1{}\, y^{2}+\t_2{}\, y^4-1), 
\quad
D=x\partial_y-(\t_1{}\,{}\, y+2\t_2{}\, y^3)\partial_x.
\end{equation*}

The associated Jordan superalgebra $J_{\s}(A,D)$ is
\begin{equation*}
\left\{
\begin{array}{ccl}
x_n\; x_m&=&x_{n+m}\\[5pt]
x_n\;y_m&=&y_{n+m}\\[5pt]
y_n\;y_m&=&x_{n+m}-\t_1{}\,{}\, x_{n+m+2}-\t_2{}\, x_{n+m+4}\\[5pt]
x_n\;a_j&=&\s a_{n+j}\\[5pt]
x_n\;b_j&=&\s b_{n+j}\\[5pt]
y_n\;a_j&=&\s b_{n+j}\\[5pt]
y_n\;b_j&=&\s (a_{n+j}-\t_1{}\,{}\, a_{n+j+2}-\t_2{}\, a_{n+j+4})\\[5pt]
a_i\;a_j&=&(j-i)y_{i+j}\\[5pt]
a_i\;b_j&=&(j-i)x_{i+j}-\t_1{}\,{}\, (j-i+1)x_{i+j+2}-\t_2{}\, (j-i+2)x_{i+j+4}\\[5pt]
b_i\;b_j&=&(j-i)(y_{i+j}-\t_1{}\,{}\, y_{i+j+2}-\t_2{}\, y_{i+j+4})\\[5pt]
\end{array}
\right.
\end{equation*}
and the Lie superalgebra $\Lf(A,D)$ is
\begin{equation*}
\left\{
\begin{array}{ccl}
[A_i,A_j]&=&  L_{i+j},\\[5pt]
[A_i, B_j] &=&H_{i+j},\\[5pt]
[B_i, B_j] &=&  L_{i+j}-\t_1{}\,{}\, L_{i+j+2}-\t_2{}\, L_{i+j+4},\\[5pt]
[L_n\; , \;A_i]&=&(i-\frac{n}{2})B_{n+i}\\[5pt]
[L_n\; , \;B_i]&=&(i-\frac{n}{2})A_{n+i}-\t_1{}\,{}\,(i-\frac{n}{2}+1)A_{n+i+2}-\t_2{}\,(i-\frac{n}{2}+2)A_{n+i+4}\\[5pt]
[H_n\; , \;A_i]&=&(i-\frac{n}{2})A_{n+i}-\t_1{}\,{}\,(i-\frac{n}{2}-\half)A_{n+i+2}-\t_2{}\,(i-\frac{n}{2}-1)A_{n+i+4}\\[5pt]
[H_n\; , \;B_i]&=&(i-\frac{n}{2})B_{n+i}-\t_1{}\,{}\,(i-\frac{n}{2}+\half)B_{n+i+2}-\t_2{}\,(i-\frac{n}{2}+1)B_{n+i+4}\\[5pt]
[L_n, L_m] &=& (m-n) H_{n+m},\\[5pt]
[L_n, H_m] &=& (m-n)L_{n+m}-\t_1{}\,{}\,(m-n+1)L_{m+n+2}-\t_2{}\,(m-n+2)L_{m+n+4},\\[5pt]
[H_n, H_m] &=& (m-n)(H_{n+m}-\t_1{}\,{}\, H_{n+m+2}-\t_2{}\, H_{n+m+4})
\end{array}
\right.
\end{equation*}

The algebra $\Jc_1(\t,2)$ used in \cite{Zhe2} (see also Section \ref{newtype})
corresponds to the particular case
$\t_1=0$, $\t_2=\t$.

\medskip
\noindent
\textbf{Acknowledgement:} 
Authors are grateful to V. Ovsienko for helpful comments and enlightening discussions.
We are also pleased to thank K. Iohara for useful references.



\end{document}